\documentclass{amsart}
\usepackage{amssymb, enumitem}
\usepackage[all]{xy}
\usepackage{hyperref, aliascnt}
\usepackage{mathtools}
\usepackage{bbm}
\def\today{\number\day\space\ifcase\month\or   January\or February\or
   March\or April\or May\or June\or   July\or August\or September\or
   October\or November\or December\fi\   \number\year}

\newaliascnt{thmCt}{lma}
\newtheorem{thm}[thmCt]{Theorem}
\aliascntresetthe{thmCt}

\newaliascnt{corCt}{lma}

\aliascntresetthe{corCt}

\newaliascnt{propCt}{lma}
\newtheorem{prop}[propCt]{Proposition}
\aliascntresetthe{propCt}

\newtheorem*{thm*}{Theorem}
\newtheorem*{cor*}{Corollary}
\newtheorem*{prop*}{Proposition}

\theoremstyle{definition}

\newaliascnt{pgrCt}{lma}

\aliascntresetthe{pgrCt}

\newaliascnt{dfCt}{lma}
\newtheorem{df}[dfCt]{Definition}
\aliascntresetthe{dfCt}

\newaliascnt{remCt}{lma}

\aliascntresetthe{remCt}

\newaliascnt{remsCt}{lma}

\aliascntresetthe{remsCt}

\newaliascnt{egCt}{lma}
\newtheorem{eg}[egCt]{Example}
\aliascntresetthe{egCt}

\newaliascnt{qstCt}{lma}

\aliascntresetthe{qstCt}

\newaliascnt{pbmCt}{lma}

\aliascntresetthe{pbmCt}

\newaliascnt{notaCt}{lma}

\aliascntresetthe{notaCt}

\newcommand{\beq}{\begin{equation}}
\newcommand{\eeq}{\end{equation}}
\newcommand{\beqa}{\begin{eqnarray*}}
\newcommand{\eeqa}{\end{eqnarray*}}
\newcommand{\bal}{\begin{align*}}
\newcommand{\eal}{\end{align*}}
\newcommand{\bi}{\begin{itemize}}
\newcommand{\ei}{\end{itemize}}
\newcommand{\be}{\begin{enumerate}}
\newcommand{\ee}{\end{enumerate}}

\newcommand{\Z}{{\mathbb{Z}}}
\newcommand{\R}{{\mathbb{R}}}
\newcommand{\C}{{\mathbb{C}}}
\newcommand{\N}{{\mathbb{N}}}

\newcommand{\B}{{\mathcal{B}}}

\newcommand{\ev}{{\mathrm{ev}}}

\newcommand{\QSL}{Q\!S\!L}
\newcommand{\SL}{S\!L}
\newcommand{\QL}{Q\!L}

\DeclareMathOperator{\Rep}{Rep}

\newcommand{\ca}{$C^*$-algebra}

\newcommand{\I}{\infty}

\title{Functoriality of group algebras acting on $L^p$-spaces}
\date{\today}

\author[Eusebio Gardella]{Eusebio Gardella}
\address{Eusebio Gardella
Department of Mathematics, Deady Hall, University of Oregon
Eugene OR, 97403, USA.}
\email{gardella@uoregon.edu}
\urladdr{http://pages.uoregon.edu/gardella/}
\author{Hannes Thiel}
\address{Hannes Thiel
Mathematisches Institut, Fachbereich Mathematik und Informatik der
Universit\"at M\"unster, Einsteinstrasse 62, 48149 M\"unster, Germany.}
\email{hannes.thiel@uni-muenster.de}
\urladdr{www.math.ku.dk/~thiel/}
\thanks{The first named author was partially supported by the D.~K. Harrison Prize from the
University of Oregon. The second named author was partially supported by the Deutsche
Forschungsgemeinschaft (SFB 878).}

\begin{document}

\begin{abstract}
We continue our study of group algebras acting on $L^p$-spaces, particularly of algebras of
$p$-pseudofunctions of locally compact groups. We
focus on the functoriality properties of these objects. We show that $p$-pseudofunctions
are functorial with respect to homomorphisms that are either injective, or whose kernel is
amenable and has finite index. We also show that the universal completion of the group algebra
with respect to representations on $L^p$-spaces, is functorial with respect to quotient maps.

As an application, we show that the algebras of $p$- and $q$-pseudofunctions on $\Z$ are isometrically
isomorphic as Banach algebras if and only if $p$ and $q$ are either equal or conjugate.

\end{abstract}
\maketitle
\tableofcontents

\section{Introduction}

Associated to a locally compact group, there are several Banach algebras that capture different aspects of its structure and representation theory.
For instance, in \cite{Her73SynthSubgps}, Herz introduced the Banach algebra of $p$-pseudofunctions of
a locally compact group $G$, for a fixed H\"{o}lder exponent $p\in[1,\infty)$. (We are thankful to
Yemon Choi and Matthew Daws for providing this reference.)
This Banach algebra is defined as the completion of the group algebra $L^1(G)$ with respect to the norm induced by the left regular representation $\lambda_p$ of $G$ on $L^p(G)$.
We denote this algebra by $F^p_\lambda(G)$, so that
\[
F^p_\lambda(G)=\overline{\lambda_p(L^1(G))} \subseteq \B(L^p(G)).
\]

In \cite{GarThie14pre:LpGpAlg}, we studied the universal completion of $L^1(G)$ for representations of $G$ on $L^p$-spaces,
which we denote by $F^p(G)$ (this algebra first appeared in \cite{Phi13arX:LpCrProd}, as the crossed
product of $G$ on the $L^p$-operator algebra $\C$).
By universality of $F^p(G)$, the identity map on $L^1(G)$ induces a contractive homomorphism
$\kappa\colon F^p(G)\to F^p_\lambda(G)$ with dense range. One of the main results of \cite{GarThie14pre:LpGpAlg},
obtained independently by Phillips, asserts that $G$ is amenable if and only if $\kappa$ is an (isometric) isomorphism.

For $p=2$, the Banach algebra $F^2(G)$ is the full group \ca{} of $G$, usually denoted $C^*(G)$, and $F^2_\lambda(G)$ is the reduced group \ca{} of $G$, usually denoted $C^*_\lambda(G)$.
The functoriality properties of the full and reduced group \ca{s} are well-understood.
Given a locally compact group $G$, a normal subgroup $N$ of $G$, and a closed subgroup $H$ of $G$,
the following results can be found in \cite{BroOza08Book}:
\be
\item[(a)]
If $G$ is discrete, then the inclusion map $H\to G$ induces natural isometric, unital homomorphism
$C^*_\lambda(H)\to C^*_\lambda(G)$;
\item[(b)]
The quotient map $G\to G/N$ induces a natural quotient homomorphism $C^*(G)\to C^*(G/N)$;
\item[(c)]
If $N$ is amenable, then the quotient map $G\to G/N$ induces a natural homomorphism
$C^*_\lambda(G)\to C^*_\lambda(G/N)$.
\ee

In this paper, we explore the extent to which these results generalize to the case $p\neq 2$.
Many techniques from $C^*$-algebra theory, such as positivity, are no longer available for Banach algebras
acting on $L^p$-spaces. In particular, some standard facts in $C^*$-algebras fail for the classes of Banach
algebras here considered.
For example, a contractive homomorphism with dense range is not necessarily surjective, and
an injective homomorphism need not be isometric.

Our results are as follows (the second one is proved in greater generality than what is reproduced below):

\be
\item
If $H$ is a subgroup of a discrete group $G$, then there is a natural isometric unital map $F^p_\lambda(H)\to F^p_\lambda(G)$ (\autoref{prop: subgroup reduced algs});
\item
If $N$ is a closed normal subgroup of a locally compact group $G$, then there is a natural contractive map
$F^p(G)\to F^p(G/N)$ with dense range (\autoref{prop: InducedQuotientMap});
\item
If $N$ is an amenable normal subgroup of a discrete group $G$, and $G/N$ is finite,
then the natural map $F^p(G)\to F^p(G/N)$ is a quotient map (\autoref{thm:amenableKernel}).
\ee

We point out that the assumption that $G/N$ be finite in (3) above is likely to be unnecessary. On the other hand,
we show in \autoref{eg:NeedAmKer}, using a result of Pooya-Hejazian in \cite{PoyHej14arX:SimpleLp},
that amenability of $N$ is necessary.

In Section 3, we apply our results to study the isomorphism type of the Banach algebras
$F^p_\lambda(\Z)$, with focus on its dependence on the H\"older exponent $p$.
We show that for $p,q\in[1,2]$, there is an isometric isomorphism between $F^p_\lambda(\Z)$ and
$F^q_\lambda(\Z)$ if and only if $p=q$.

Further applications of the results of this paper appear in \cite{GarThi14arX:LpGenInvIsom}.
\newline

Throughout, we will assume that all measure spaces are $\sigma$-finite, and that all Banach spaces
are separable.
Consistently, all locally compact groups will be assumed to be second countable, and will be endowed
with a left Haar measure.

We take $\N=\{1,2,\ldots\}$. For $n$ in $\N$ and $p\in [1,\I]$,
we write $\ell^p_n$ in place of $\ell^p(\{1,\ldots,n\})$, and we write $\ell^p$ in place of $\ell^p(\Z)$.

Let $E$ be a Banach space. We write $\B(E)$ for the Banach algebra of bounded linear operators on $E$.
For $p\in (1,\I)$, we denote by $p'$ its conjugate (H\"older) exponent, which satisfies $\frac{1}{p}+\frac{1}{p'}=1$.
\newline

\textbf{Acknowledgements.} Part of this work was completed while the authors were attending the Thematic Program on Abstract Harmonic Analysis, Banach and Operator Algebras at the Fields Institute in January-June 2014, and while the second named author was visiting the University of Oregon in July and August 2014. The hospitality of the Fields Institute and the University of Oregon are gratefully acknowledged.

The authors would like to thank Chris Phillips and Nico Spronk for helpful conversations,
as well as Antoine Derighetti and Bill Johnson for electronic correspondence.

\section{Functoriality properties}

In this section, we study the extent to which group homomorphisms induce Banach algebra homomorphisms between the respective group operator algebras we studied in \cite{GarThie14pre:LpGpAlg}.
As in the case of group $C^*$-algebras, these completions are not functorial with respect to arbitrary group homomorphisms. Section 3 contains an application of these results, particularly of \autoref{thm:amenableKernel}:
the Banach algebras $F^p_\lambda(\Z)$ and $F^q_\lambda(\Z)$ are isometrically
isomorphic if and only if either $p=q$ or $p=q'$; see \autoref{thm: FpZ not isom}.

We begin by recalling some definitions and results from \cite{GarThie14pre:LpGpAlg}.

\begin{df}\label{df:Classes}
Let $E$ be a (separable) Banach space.
\be
\item
We say that $E$ is an \emph{$L^p$-space} if there exists a $\sigma$-finite measure
space $(X,\mu)$ such that $E$ is isometrically isomorphic to $L^p(X,\mu)$.
We denote by $\mathcal{L}^p$ the class of (separable) $L^p$-spaces.
\item
We say that $E$ is an \emph{$\SL^p$-space} if there exists an $L^p$-space $F$ such that $E$ is isometrically isomorphic to a closed subspace of $F$.
We denote by $\mathcal{SL}^p$ the class of (separable) $\SL^p$-spaces.
\item
We say that $E$ is a \emph{$\QL^p$-space} if there exists an $L^p$-space $F$ such that $E$ is isometrically isomorphic to a quotient of $F$.
We let $\mathcal{QL}^p$ denote the class of (separable) $\QL^p$-spaces.
\item
We say that $E$ is a \emph{$\QSL^p$-space} if there exists an $\SL^p$-space $F$ such that $E$ is isometrically isomorphic to a quotient of $F$.
We let $\mathcal{QSL}^p$ denote the class of (separable) $\QSL^p$-spaces.
\ee

If $\mathcal{E}$ is any of the classes considered above, we denote by $\Rep_\mathcal{E}(G)$ the class
of all contractive representations of $L^1(G)$ on Banach spaces in $\mathcal{E}$. We denote by
$F_\mathcal{E}(G)$ the completion of $L^1(G)$ in the norm given by
\[\|f\|_\mathcal{E}=\sup\left\{\|\pi(f)\|\colon \pi\in \Rep_\mathcal{E}(G)\right\}\]
for $f\in L^1(G)$.

The algebra of $p$-pseudofunctions on $G$, denoted by $F^p_\lambda(G)$, is the completion
of $L^1(G)$ in the norm
\[\|f\|_{F^p_\lambda(G)}=\|\lambda_p(f)\|_{\B(L^p(G))}\]
for $f\in L^1(G)$.
\end{df}

By universality of the objects constructed in \autoref{df:Classes}, there exist canonical maps
making the diagram
\[
\xymatrix{
& { F^p_{\mathrm{S}}(G) } \ar[dr]^{\kappa_{\mathrm{S}}} \\
{ F^p_{\mathrm{QS}}(G) } \ar[ur]^{\kappa_{\mathrm{QS},\mathrm{S}}}
\ar[dr]_{\kappa_{\mathrm{QS},\mathrm{Q}}}
& & { F^p(G) }\ar[r]^\kappa & F^p_\lambda(G) \\
& { F^p_{\mathrm{Q}}(G) } \ar[ur]_{\kappa_{\mathrm{Q}}}
}
\]
commute; see the comments after Remark 2.14 in \cite{GarThie14pre:LpGpAlg}.
These maps have dense range, since they are suitable extensions of the identity map on $L^1(G)$.

The algebra $F^p_\lambda(G)$ of $p$-pseudofunctions, together with the related Banach algebras of
$p$-pseudomeasures $PM_p(G)$ and $p$-convolvers $CV_p(G)$ on $G$, have been studied by a number
of authors; see for example \cite{Her73SynthSubgps}, \cite{Der11ConvOps}, and \cite{NeuRun09ColumnRowQSL}.
Also, the algebra $F^p_{\mathrm{QS}}(G)$ can be seen to be isometrically isomorphic to the algebra
of universal $p$-pseudofunctions $UPF_p(G)$ on $G$, introduced by Runde in \cite{Run05ReprQSL}.

The following is part of Theorem~3.7 in \cite{GarThie14pre:LpGpAlg}.

\begin{thm}
\label{thm:AmenTFAE}
Let $G$ be a locally compact group, and let $p\in(1,\infty)$.
The following are equivalent:
\begin{enumerate}
\item
The group $G$ is amenable.
\item
With $\mathcal{E}$ denoting any of the classes $\mathcal{QSL}^p$, $\mathcal{QL}^p$, $\mathcal{SL}^p$, or $\mathcal{L}^p$, the canonical map $F_\mathcal{E}(G)\to F^p_\lambda(G)$ is an isometric isomorphism.
\end{enumerate}
\end{thm}

We now turn to functoriality of these Banach algebras. The case $p=2$ of the following result is proved,
for example, as Proposition~2.5.9 in \cite{BroOza08Book}.

\begin{prop}\label{prop: subgroup reduced algs}
Let $p\in[1,\infty)$, let $G$ be a discrete group and let $H$ be a subgroup of $G$.
Then the canonical inclusion $\iota\colon H\hookrightarrow G$ induces an isometric
embedding $F^p_\lambda(H)\to F^p_\lambda(G)$.
\end{prop}
\begin{proof}
We denote also by $\iota\colon \C[H]\to \C[G]$ the induced algebra homomorphism.
Let $\lambda^G_p\colon\C[G]\to \B(\ell^p(G))$ and $\lambda^H_p\colon\C[H]\to \B(\ell^p(H))$ denote the left regular
representations of $G$ and $H$, respectively.
Then $\lambda_p^G\circ\iota$ is conjugate, via an invertible isometry, to a multiple of $\lambda_p^H$.
More precisely, let $Q$ be a subset of $G$ containing exactly one element from each coset in $G/H$.
Then there is a canonical isometric isomorphism
\[
\ell^p(G) \cong \bigoplus_{x\in Q} \ell^p(xH).
\]
The representation $\lambda_p^G\circ\iota\colon\C[H]\to \B(\ell^p(G))$ leaves each of the subspaces $\ell^p(xH)\subseteq\ell^p(G)$
invariant, and hence
\[
\lambda_p^G\circ\iota \cong \bigoplus_{x\in Q} \lambda^H_p.
\]
It follows that
\[
\|\iota(f)\|_{F^p_\lambda(G)}
= \| (\lambda_p^G\circ\iota)(f) \|
= \left\| \bigoplus_{x\in Q} \lambda^H_p(f) \right\|
= \max_{x\in Q} \| \lambda^H_p(f) \|
= \|f\|_{F^p_\lambda(H)}
\]
for every $f\in\C[H]$. Thus, the canonical map $\iota\colon F^p_\lambda(H)\to F^p_\lambda(G)$ is isometric, as desired.
\end{proof}

We need some notation for the next result. If $G$ is a locally compact group and $N$ is a closed normal subgroup, then
there is a canonical surjective contractive homomorphism $\psi_N\colon L^1(G)\to L^1(G/N)$ which satisfies
\[\int_{G/N}\psi_N(f)(sN)\ d(sN)=\int_Gf(s)\ ds\]
for all $f$ in $L^1(G)$; see Theorem~3.5.4 in \cite{ReiSte00HarmAna}.

\begin{prop}
\label{prop: InducedQuotientMap}
Let $p\in[1,\infty)$, let $G$ be a locally compact group, let $N$ be a closed normal subgroup of $G$, and let $\pi\colon G\to G/N$ be the canonical quotient map.
If $\mathcal{E}$ denotes any of the classes $\mathcal{QSL}^p$, $\mathcal{SL}^p$, $\mathcal{QL}^p$, or $\mathcal{L}^p$,
then $\pi$ induces a natural contractive map $F_\mathcal{E}(G)\to F_\mathcal{E}(G/N)$ with dense range.
\end{prop}
\begin{proof}
Let $\mathcal{E}$ denote any of the classes $\mathcal{QSL}^p$, $\mathcal{SL}^p$, $\mathcal{QL}^p$, or $\mathcal{L}^p$.
Denote by $\psi_N\colon L^1(G)\to L^1(G/N)$ the surjective contractive homomorphism described in the comments above.
Given $f\in L^1(G)$, we have
\begin{align*}
\|\psi_N(f)\|_{\mathcal{E}} &= \sup\{ \|(\omega\circ\psi_N)(f)\| \colon \omega \in \operatorname{Rep}_H(\mathcal{E}) \} \\
&\leq \sup\{ \| \rho(f)\| \colon \rho \in \operatorname{Rep}_G(\mathcal{E}) \}\\
&= \|f\|_{\mathcal{E}}.
\end{align*}
It follows that $\psi_N$ extends to a contractive homomorphism $F_\mathcal{E}(G)\to F_\mathcal{E}(G/N)$ with dense range.
\end{proof}

The above proposition shows that the universal completions of $L^1(G)$ are
functorial with respect to surjective group homomorphisms.
When $p$ is not equal to 1 or 2, it is not clear whether the resulting homomorphism $F^p(G)\to F^p(G/N)$
is a quotient map, or even if it is surjective.
In the following theorem, we prove that this is indeed the case whenever $N$ is amenable and $G/N$ is finite.

\begin{thm}
\label{thm:amenableKernel}
Let $G$ be a discrete group, let $p\in [1,\infty)$, and let $N$ be an amenable normal
subgroup of $G$ such that $G/N$ is finite.
Then the canonical map $G\to G/N$ induces a natural quotient homomorphism
$F_\lambda^p(G)\to F_\lambda^p(G/N)$.
\end{thm}
\begin{proof}
We establish some notation first:
\bi
\item
For $s\in G$, we write $u_s$ for the corresponding element in $\C[G]$, and $\delta_s\in \ell^p(G)$ for the corresponding basis element;
\item
For $s\in G$, we write $v_{sN}$ for the corresponding element in $\C[G/N]$, and $\delta_{sN}\in \ell^p(G/N)$ for the corresponding basis element;
\item
For $n\in N$, we write $w_n$ for the corresponding element in $\C[N]$, and $\delta_n\in \ell^p(N)$ for the corresponding basis element;
\item
We write $\pi\colon \C[G]\to \C[G/N]$ for the map given by $u_s\mapsto v_{sN}$ for $s\in G$.
\ei
Fix a section $\sigma\colon G/N\to G$, and define an isometric isomorphism
\[
\varphi\colon \ell^p(G/N)\otimes \ell^p(N)\to \ell^p(G)
\]
by $\varphi(\delta_{sN}\otimes \delta_n)=\delta_{\sigma(sN)n}$ for $s\in G$ and $n\in N$.
Let
\[
\Phi\colon \B(\ell^p(G/N))\otimes \B(\ell^p(N))\to \B(\ell^p(G))
\]
be the isometric isomorphism given by $\Phi(x)=\varphi\circ x\circ\varphi^{-1}$ for $x\in\B(\ell^p(G/N))\otimes \B(\ell^p(N))$.
It is a routine exercise to check that
\[
\Phi(v_{sN}\otimes w_n)(\delta_{\sigma(tN)m})=\delta_{\sigma(stN)nm}
\]
for all $s,t\in G$ and all $n,m\in N$.

Let $f$ be an element in $\C[G/N]$.
We want to show that
\[
\|f\|=\inf \{\|\widetilde{f}\|\colon \widetilde{f}\in \C[G], \pi(\widetilde{f})=f\}.
\]
For this, it is enough to find sequences $(f_k)_{k\in\N}$ in $\B(\ell^p(G))$ (but not necessarily in $\C[G]$) and $(\widetilde{f}_k)_{k\in\N}$ in $\C[G]\subseteq \B(\ell^p(G))$, such that
\be
\item
$\|f_k\|\leq \|f\|$ for all $k\in \N$;
\item
$\pi(\widetilde{f}_k)=f$ for all $k\in \N$; and
\item
$\lim\limits_{k\to\I}\|\widetilde{f}_k-f_k\|=0$.
\ee

Let $S\subseteq G$ be a finite set such that $f$ can be written as a finite linear combination
$f=\sum\limits_{s\in S}a_{sN}v_{sN}$, where $a_{sN}$ is a complex number for $s\in S$.
Using amenability of $N$, choose a F{\o}lner sequence $(F_k)_{k\in\N}$ of finite subsets of $N$ satisfying
\[
\lim_{k\to \I}\frac{|F_k\triangle F_kx|}{|F_k|}=0
\]
for all $x\in N$.
For $k\in \N$, set $T_k=\frac{1}{|F_k|}\sum\limits_{n\in F_k} w_n$, which is an element in $\C[N]$.\\
\ \\
\indent Let $k\in\N$. We claim that $\|T_k\|_{F^p(N)}=1$.

Note that $T_k$ is a linear combination of the canonical generating invertible isometries with positive coefficients (the coefficients are all either $\frac{1}{|F_k|}$ or $0$).
It follows from Theorem~4.19 in \cite{Pat88Amen} that $\|T_k\|_p=\|T_k\|_2$.
Furthermore, the equivalence between (1) and (8) in Theorem~2.6.8 in \cite{BroOza08Book} shows that $\|T_k\|_2=1$.
The claim is proved.\\
\ \\
\indent Fix $k\in\N$, and set
\[
f_k=T_k\circ\Phi(f\otimes 1),
\]
which is an element in $\B(\ell^p(G))$.
(Note that $f_k$ will not
in general belong to the group algebra $\C[G]$.)
Basic properties of $p$-tensor products give $\|\Phi(f\otimes 1)\|=\|f\|$, and hence
$\|f_k\| \leq\|T_k\|\cdot\|f\|=\|f\|$, so condition (1) above is satisfied.
Set
\[
\widetilde{f}_k = \frac{1}{|F_k|}\sum_{s\in S}\sum_{n\in F_k} a_{sN} u_{n\sigma(sN)},
\]
which is an element in $\C[G]\subseteq \B(\ell^p(G))$.
It is clear that $\pi(\widetilde{f}_k)=f$, so condition (2) above
is also satisfied.
We need to check (3).
With $M=\max\limits_{s\in S}|a_{sN}|$, we have
\begin{align*}
\|\widetilde{f}_k-f_k\|_p
&=\frac{1}{|F_k|} \left\| \sum_{s\in S} a_{sN}\sum_{n\in F_k} u_{n\sigma(sN)}-u_n\Phi(v_{sN}\otimes 1) \right\|_p \\
&\leq M \left( \frac{1}{|F_k|} \left\| \sum_{n\in F_k} u_{n\sigma(sN)}-u_n\Phi(v_{sN}\otimes 1) \right\|_p \right).
\end{align*}
Given $s$ in $G$, it is therefore enough to show that
\[
\lim_{k\to\I} \frac{1}{|F_k|} \left\| \sum_{n\in F_k} u_{n\sigma(sN)}-u_n\Phi(v_{sN}\otimes 1) \right\|_p=0.
\]

Fix $s$ in $G$ and set
\[
\theta_k=\frac{1}{|F_k|}\sum\limits_{n\in F_k} u_{n\sigma(sN)}-u_n\Phi(v_{sN}\otimes 1),
\]
regarded as an operator on $c_c(G)$.
It is immediate that for $q\in [1,\I]$, the operator $\theta$ extends to a bounded operator
$\theta_k^{(q)}$ on $\ell^q(G)$ with $\left\|\theta_k^{(q)}\right\|_q\leq 2$, and the Riesz-Thorin Interpolation Theorem gives
\[
\left\|\theta_k^{(p)}\right\|_p\leq \left\|\theta_k^{(1)}\right\|^{\frac{1}{p}}_1\left\|\theta_k^{(\I)}\right\|^{\frac{1}{p'}}_\I\leq 2\left\|\theta_k^{(1)}\right\|^{\frac{1}{p}}_1.
\]
It therefore suffices to show that $\lim\limits_{k\to \I}\left\|\theta_k^{(1)}\right\|_1=0$.

Let $c\colon G\times G\to N$ be the 2-cocycle given by
\[
c(t,r)\sigma(tN)\sigma(rN)=\sigma(trN)
\]
for all $t$ and $r$ in $G$. Since $G/N$ is finite, the image $\mathrm{Im}(c)$ of the $2$-cocycle $c$ is a finite
subset of $N$. Given $t\in G$ and $m\in N$, we have
\begin{align*}
\theta^{(1)}_k(\delta_{\sigma(tN)m})
&= \frac{1}{|F_k|}\sum_{n\in F_k}
\left( \delta_{n\sigma(sN)\sigma(tN)m} - \delta_{n\sigma(stN)m} \right) \\
&= \frac{1}{|F_k|}\sum_{n\in F_k}
\left( \delta_{n\sigma(sN)\sigma(tN)m} - \delta_{nc(s,t)\sigma(sN)\sigma(tN)m} \right) \\
\end{align*}
Thus,
\begin{align*}
\left\| \theta_k^{(1)} \right\|_1
&= \sup_{t\in G}\sup_{m\in N}
\left\| \theta_k^{(1)}(\delta_{\sigma(tN)m}) \right\|_1 \\
&= \sup_{x\in \mathrm{Im}(c)} \sup_{\substack{t\in G\colon \\ c(s,t)=x }}\sup_{m\in N}
\frac{1}{|F_k|} \left\| \sum_{n\in F_k} \delta_{n\sigma(sN)\sigma(tN)m} -\delta_{nx\sigma(sN)\sigma(tN)m} \right\|_1 \\
&= \sup_{x\in \mathrm{Im}(c)} \sup_{\substack{t\in G\colon \\ c(s,t)=x }}\sup_{m\in N}
\frac{\left| F_k\sigma(sN)\sigma(tN)m\ \triangle\ F_kx\sigma(sN)\sigma(tN)m \right|}{|F_k|} \\
&= \sup_{x\in \mathrm{Im}(c)} \sup_{\substack{t\in G\colon \\ c(s,t)=x }}\sup_{m\in N}
\frac{\left| F_k \triangle F_kx \right|}{|F_k|}\\
&= \sup_{x\in \mathrm{Im}(c)} \frac{\left| F_k \triangle F_kx \right|}{|F_k|}
\end{align*}
Since $(F_k)_k$ is a F{\o}lner sequence and $\mathrm{Im}(c)$ is finite, the above computation implies that
$\lim\limits_{k\to\infty} \left\| \theta_k^{(1)} \right\|_1 = 0,$ as desired.
This finishes the proof.
\end{proof}

We point out that the assumption that $N$ be amenable is necessary in the theorem above,
at least when $p\neq 1$, as the next example shows.

\begin{eg} \label{eg:NeedAmKer}
Fix $p\in (1,\I)$.
Let $\mathbb{F}_2$ denote the free group on two generators, and let $N$ be a normal subgroup of
$\mathbb{F}_2$ such that $\mathbb{F}_2/N$ is isomorphic to $\Z_2$. The quotient map $\pi\colon \mathbb{F}_2\to \Z_2$
does not induce a quotient map $F^p_\lambda(\mathbb{F}_2)\to F^p_\lambda(\Z_2)$, since $F^p_\lambda(\mathbb{F}_2)$ is simple
by Corollary~3.11 in \cite{PoyHej14arX:SimpleLp}.\end{eg}

On the other hand, we suspect that no condition on $G/N$ is needed for the conclusion of \autoref{thm:amenableKernel} to hold (and that, in particular, the group $G$ need not be amenable), but we have not been able to prove the more general statement.
For $p=2$, this can be proved as follows. Since $N$ is amenable, its trivial representation is weakly contained
in its left regular representation (see Theorem~2.6.8 in \cite{BroOza08Book}). Using the fact that the induction functor
preserves weak containment of representations, this shows
that the left regular representation of $G/N$ is weakly contained in the left regular representation of $G$.
By the comments at the beginning of Appendix~D in \cite{BroOza08Book}, this implies that there is a $\ast$-homomorphism
$C^*_\lambda(G)\to C^*_\lambda(G/N)$ with dense range.
Finally, basic $C^*$-algebra theory (for example, the fact that $\ast$-homomorphisms have closed range)
shows that this map is indeed a quotient map.

There is an alternative proof of this fact using F\o lner sets, similarly to what we did in the proof of
\autoref{thm:amenableKernel}, but the argument also involves the GNS construction, which so far has no analog
in the context of $\mathcal{L}^p$-operator algebras.

\section{An application: When is \texorpdfstring{$F^p(\Z)$}{Fp(Z)} isomorphic to
\texorpdfstring{$F^q(\Z)$}{Fq(Z)}?}

The goal of this section is to show that for $p$ and $q$ in $[1,\infty)$, there is an isometric isomorphism between
$F^p(\Z)$ and $F^q(\Z)$ if and only if either $p=q$ or $\frac{1}{p}+\frac{1}{q}=1$. The strategy will be to use \autoref{thm:amenableKernel}, Proposition~3.13 in \cite{GarThie14pre:LpGpAlg}, and the fact that every homeomorphism of $S^1$ must map a pair of antipodal points to antipodal points, to reduce this to the case when the group is $\Z_2$, where things can be proved more directly. The fact that the spectrum of $F^p(\Z)$ is the circle is crucial in
our proof, and we do not know how to generalize these methods to deal with, for example, $\Z^2$.

We begin by looking at the group $\mathcal{L}^p$-operator algebra of a finite cyclic group.

\begin{eg} \label{eg: FpZn}
Let $n$ in $\N$ and let $p\in [1,\I)$. Consider the group $\mathcal{L}^p$-operator algebra $F^p(\Z_n)$ of $\Z_n$. Then $F^p(\Z_n)$ is the Banach
subalgebra of $\B(\ell^p_n)$ generated by the cyclic shift of order $n$
\[s_n= \left( \begin{array}{ccccc}
0 &  &  &  & 1 \\
1 & 0 & &  &  \\
& \ddots & \ddots &  &  \\
&  & \ddots & 0 &  \\
&  &  & 1 & 0 \end{array} \right).\]
(The algebra $\B(\ell^p_n)$ is $M_n$ with the $L^p$-operator norm.) It is easy to check that $F^p(\Z_n)$ is isomorphic, as a complex algebra, to $\C^n$,
but the canonical embedding $F^p(\Z_n)\hookrightarrow M_n$ is not as diagonal matrices.

It turns out that computing the norm of a vector in $\C^n\cong F^p(\Z_n)$ is challenging for $p$ different
from 1 and 2, essentially because computing $p$-norms of matrices that are not diagonal is difficult.
Indeed, set
$\omega_n=e^{\frac{2\pi i}{n}}$, and set
\[u_n= \frac{1}{\sqrt{n}} \left( \begin{array}{ccccc}
1 & 1 & 1 & \cdots & 1 \\
1 &\omega_n &\omega_n^2 & \cdots &\omega_n^{n-1} \\
1 &\omega_n^2 &\omega_n^4 & \cdots &\omega_n^{2(n-1)} \\
\vdots & \vdots & \vdots & \ddots & \vdots \\
1 &\omega_n^{n-1} &\omega_n^{2(n-1)} & \cdots &\omega_n^{(n-1)^2} \end{array} \right).\]
If $\xi=(\xi_1,\ldots,\xi_n)\in \C^n$, then its norm as an element in $F^p(\Z_n)$ is
\[ \|\xi\|_{F^p(\Z_n)}=\left\| u_n \left( \begin{array}{ccccc}
\xi_1 &  &  &    \\
 & \xi_2 & &   \\
 &  & \ddots &    \\
&  &  &  \xi_n  \end{array} \right)u_n^{-1}\right\|_p.\]
\indent The matrix $u_n$ is a unitary (in the sense that its conjugate transpose is its inverse), and hence $\|\xi\|_{F^2(\Z_n)}=\|\xi\|_\infty$. The norm
on $F^2(\Z_n)$ is therefore well-understood and easy to compute. On the other hand, if $1\leq p\leq q\leq 2$,
then $\|\cdot\|_{F^q(\Z_n)}\leq \|\cdot\|_{F^{p}(\Z_n)}$ by Corollary~3.20 in \cite{GarThie14pre:LpGpAlg}.
In particular, the norm $\|\cdot\|_{F^p(\Z_n)}$ always dominates the norm
$\|\cdot\|_\I$.\end{eg}

Computing the automorphism group of $F^p(\Z_n)$ is not easy when $p\neq 2$, since not every permutation of the
coordinates of $\C^n\cong F^p(\Z_n)$ induces an isometric isomorphism. Our next result asserts that the cyclic
shift on $\C^n$ is isometric.

\begin{prop} \label{prop: shift invariance norm on FpZn}
Let $n$ in $\N$ and let $p$ in $[1,\I)$. Denote by $\tau\colon\C^n \to \C^n$ the cyclic forward shift, this is,
$$\tau(x_0,\ldots,x_{n-1})=(x_{n-1},x_0,\ldots,x_{n-2})$$
for all $(x_0,\ldots,x_{n-1})\in \C^n$. Then $\tau\colon F^p(\Z_n)\to F^p(\Z_n)$ is an isometric isomorphism.
\end{prop}
\begin{proof}
We follow the notation from \autoref{eg: FpZn}, except that we drop the subscript $n$ everywhere, so we write $u$
in place of $u_n$, and we write $s$ in place of $s_n$. (We still denote $\omega_n=e^{\frac{2\pi i}{n}}$.)\\
\indent For $x$ in $\C^n$, let $\mbox{d}(x)$ denote the diagonal $n\times n$ matrix with $\mbox{d}(x)_{j,k}=\delta_{j,k}x_j$ for $0\leq j,k\leq n-1$.
Denote by $\rho\colon\C^n\to M_n$ the algebra homomorphism given by $\rho(x)= u\mbox{d}(x)u^{-1}$ for $x\in\C^n$.
Then
\[
\|x\|_{F^p(\Z_n)} =\|\rho(x)\|_p = \|u\mbox{d}(x)u^{-1}\|_p
\]
for all $x\in \C^n$.\\
\indent Set $\omega=(1,\omega_n^1,\ldots,\omega_n^{n-1})\in\C^n$, and denote by $\overline{\omega}$ its (coordinatewise) conjugate.
Given $x\in\C^n$, it is easy to check that
\[
\mathrm{d}(\tau(x)) = s \mbox{d}(x) s^{-1}, \ \  us = \mbox{d}(\omega)u, \ \ \text{ and }\ \ s^{-1}u^{-1} = u^{-1} \mbox{d}(\overline{\omega}).
\]
It follows that
\[
\|\tau(x)\|_{F^p(\Z_n)} =\|u \mathrm{d}(\tau(x))u^{-1}\|_p =\|u s \mbox{d}(x) s^{-1} u^{-1}\|_p =\|\mbox{d}(\omega) u \mbox{d}(x) u^{-1} \mbox{d}(\overline{\omega})\|_p.
\]
Since $\mbox{d}(\omega)$ and $\mbox{d}(\overline{\omega})$ are isometries in $\B(\ell^p_n)$, we conclude that
\[
\|\tau(x)\|_{F^p(\Z_n)}
=\|\mbox{d}(\omega) u \tau(x) u^{-1} \mbox{d}(\overline{\omega})\|_p
=\|u \mbox{d}(x) u^{-1}\|_p
=\|x\|_{F^p(\Z_n)},
\]
as desired. \end{proof}

The fact that $F^p(\Z_2)$ is isometrically isomorphic to $F^q(\Z_2)$ only in the trivial cases can be shown directly
by computing the norm of a specific element. We do not know whether a similar computation can be done for other cyclic
groups. However, knowing this for just $\Z_2$ is enough to prove \autoref{thm: FpZ not isom}.

\begin{prop}\label{prop: FpZ2 not isom}
Let $p$ and $q$ be in $[1,\infty)$. Then $F^p(\Z_2)$ is isometrically isomorphic to $F^{q}(\Z_2)$ if and only if
either $p=q$ of $\frac{1}{p}+\frac{1}{q}=1$.
\end{prop}
\begin{proof}
The ``if'' implication follows from Proposition 2.17 in \cite{GarThie14pre:LpGpAlg}.
We proceed to show the ``only if'' implication.

Given $r$ in $[1,\infty)$, we claim that
\[\|(1,i)\|_{F^r(\Z_2)}=2^{\left|\frac{1}{r}-\frac{1}{2}\right|}.\]

By Proposition~2.17 in \cite{GarThie14pre:LpGpAlg}, the quantity on the left-hand side remains unchanged if one replaces $r$
with its conjugate exponent. Since the same holds for the quantity on the right-hand side, it follows that it is
enough to prove the claim for $r$ in $[1,2]$.\\
\indent  Define a continuous function $\gamma\colon[1,2]\to \R$ by $\gamma(r)=\|(1,i)\|_{F^r(\Z_2)}$ for $r$ in $[1,2]$.
Let $a$ be the matrix
\[
a=\frac{1}{2} \begin{pmatrix}
1+i & 1-i \\
1-i & 1+i
\end{pmatrix}.
\]
Then $\gamma(r)= \left\| a \right\|_r$ for all $r\in [1,2]$.
The values of $\gamma$ at $r=1$ and $r=2$ are easy to compute, and we have $\gamma(1)=\|a\|_1=2^{\frac{1}{2}}$ and $\gamma(2)=\|a\|_2=1$.
Fix $r\in [1,2]$ and let $\theta$ in $(0,1)$ satisfy
$$\frac{1}{r}=\frac{1-\theta}{1} + \frac{\theta}{2}.$$
Using the Riesz-Thorin Interpolation Theorem between $r_0=1$ and $r_1=2$, we conclude that
\[
\gamma(r) \leq\gamma(1)^{1-\theta}\cdot\gamma(2)^{\theta}=2^{\frac{1}{2}(\frac{2}{r}-1)}\cdot 1= 2^{\frac{1}{r}-\frac{1}{2}}.\]

For the converse inequality, fix $r$ in $[1,2]$ and consider the vector $x=\left(\begin{smallmatrix}1 \\ 0 \end{smallmatrix}\right)$ in
$\ell^r_2$. Then $\|x\|_r=1$ and
$ax = \frac{1}{2}(\begin{smallmatrix}
1+i \\ 1-i
\end{smallmatrix})$.
We compute:
\[
\left\|\frac{1}{2}\begin{pmatrix}
1+i \\ 1-i
\end{pmatrix}\right\|_r
= \frac{1}{2} ( |1+i|^r + |1-i|^r )^{\frac{1}{r}}
= 2^{(\frac{1}{r}-\frac{1}{2})}.
\]
We conclude that
\[
\gamma(r)
=\left\| a \right\|_r
\geq \frac{\left\| ax\right\|_r}{\|x\|_r}
= 2^{(\frac{1}{r}-\frac{1}{2})}.
\]
This shows that $\gamma(r)=2^{(\frac{1}{r}-\frac{1}{2})}$ for $r\in[1,2]$, and the claim follows. \\
\ \\
\indent Now let $p$ and $q$ be in $[1,\infty)$ and let $\varphi\colon F^p(\Z_2)\to F^{q}(\Z_2)$ be an isometric isomorphism.
Since $\varphi$ is an algebra isomorphism, we must have either $\varphi(x,y)=(x,y)$ or $\varphi(x,y)=(y,x)$ for all $(x,y)\in\C^2$.
By \autoref{prop: shift invariance norm on FpZn}, the flip $(x,y)\mapsto (y,x)$ is an isometric isomorphism of
$F^{q}(\Z_2)$, so we may assume that $\varphi$ is the identity map on $\C^2$. It follows that $\|(1,i)\|_{F^p(\Z_2)}
= \|(1,i)\|_{F^{q}(\Z_2)}$, so $\left|\frac{1}{p}-\frac{1}{2}\right| = \left|\frac{1}{q}-\frac{1}{2}\right|$.
We conclude that either
$p=q$ or $\frac{1}{p}+\frac{1}{q}=1$, so the proof is complete.
\end{proof}

We are now ready to show that for $p$ and $q$ in $[1,\infty)$, the algebras $F^p(\Z)$
and $F^{q}(\Z)$ are (abstractly) isometrically isomorphic only in the trivial cases $p=q$ and $\frac{1}{p}+\frac{1}{q}=1$. (Compare this with part (2) of Corollary~3.20 in \cite{GarThie14pre:LpGpAlg},
where only the canonical homomorphism is considered.)

\begin{thm}\label{thm: FpZ not isom}
Let $p$ and $q$ be in $[1,\infty)$.
Then $F^p(\Z)$ is isometrically isomorphic to $F^q(\Z)$ if and only if either $p=q$ or $\frac{1}{q}+\frac{1}{q}=1$.
\end{thm}
\begin{proof}
The ``if'' implication follows from Proposition 2.17 in \cite{GarThie14pre:LpGpAlg}. Let us show the converse.

Recall that the maximal ideal spaces of $F^p(\Z)$ and $F^q(\Z)$ are canonically homeomorphic to $S^1$ by
Proposition~3.13 in \cite{GarThie14pre:LpGpAlg}.
We let $\Gamma_p\colon F^p(\Z)\to C(S^1)$ denote the Gelfand transform, which sends the generator $u\in F^p(\Z)$ to the canonical inclusion $\iota$ of $S^1$ into $\C$.

Let $\varphi\colon F^p(\Z)\to F^q(\Z)$ be an isometric isomorphism.
Then $\varphi$ induces a homeomorphism $f\colon S^1\to S^1$ that maps $z$ in $S^1$ to the unique point $f(z)$ in $S^1$
that satisfies
$$\ev_z\circ\varphi = \ev_{f(z)} \colon F^p(\Z)\to\C.$$
It is a classical result in point-set topology that there must exist $\zeta$ in $S^1$ such that $f(-\zeta)=-f(\zeta)$.
Denote by $\pi_p\colon F^p(\Z)\to F^p(\Z_2)$ and $\pi_q\colon F^p(\Z)\to F^q(\Z_2)$ the canonical homomorphisms associated with
the surjective map $\Z\to \Z_2$.
Then $\pi_p$ and $\pi_q$ are quotient maps by \autoref{thm:amenableKernel}.
Let $\omega_\zeta\colon F^p(\Z)\to F^p(\Z)$ be the isometric isomorphism induced by multiplying by $\zeta$ the canonical
generator in $F^p(\Z)$ corresponding to $1\in\Z$. Analogously, let $\omega_{f(\zeta)}\colon F^q(\Z)\to F^q(\Z)$ be the
isometric isomorphism induced by multiplying by $f(\zeta)$ the canonical generator in $F^q(\Z)$.
Then the following diagram is commutative:

\begin{center}
\makebox{
\xymatrix{ & C(S^1) \ar[r]^-{f^*} & C(S^1) & \\
F^p(\Z) \ar[d]_{\pi_p}  & F^p(\Z)\ar[l]_-{\omega_\zeta} \ar[r]_-{\varphi} \ar[u]^-{\Gamma_p}
& F^q(\Z) \ar[u]_-{\Gamma_q} \ar[r]^-{\omega_{f(\zeta)}} & F^q(\Z) \ar[d]^{\pi_q} \\
F^p(\Z_2)\ar@{-->}_-{\widehat{\psi}}[rrr] & & &  F^q(\Z_2).
}}
\end{center}

Define a homomorphism $\psi\colon F^p(\Z)\to F^q(\Z)$ by
$$\psi=\omega_{f(\zeta)}\circ\varphi\circ\omega_\zeta^{-1}.$$
Then $\psi$ is an isometric isomorphism. One checks that $\psi$ maps the kernel of $\pi_p$ onto the kernel of $\pi_q$.
It follows that $\psi$ induces an isometric isomorphism $\widehat{\psi}\colon F^p(\Z_2)\to F^q(\Z_2)$.
By \autoref{prop: FpZ2 not isom}, this implies that $p$ and $q$ are either equal or conjugate, as desired.
\end{proof}


\providecommand{\bysame}{\leavevmode\hbox to3em{\hrulefill}\thinspace}
\providecommand{\noopsort}[1]{}
\providecommand{\mr}[1]{\href{http://www.ams.org/mathscinet-getitem?mr=#1}{MR~#1}}
\providecommand{\zbl}[1]{\href{http://www.zentralblatt-math.org/zmath/en/search/?q=an:#1}{Zbl~#1}}
\providecommand{\jfm}[1]{\href{http://www.emis.de/cgi-bin/JFM-item?#1}{JFM~#1}}
\providecommand{\arxiv}[1]{\href{http://www.arxiv.org/abs/#1}{arXiv~#1}}
\providecommand{\doi}[1]{\url{http://dx.doi.org/#1}}
\providecommand{\MR}{\relax\ifhmode\unskip\space\fi MR }
\providecommand{\MRhref}[2]{%
  \href{http://www.ams.org/mathscinet-getitem?mr=#1}{#2}
}
\providecommand{\href}[2]{#2}

\end{document}